\documentclass{birkjour}

\usepackage{amssymb,amsmath,newlfont,enumerate}

 \newtheorem{theorem}{Theorem}[section]
 \newtheorem{corollary}[theorem]{Corollary}
 \newtheorem{lemma}[theorem]{Lemma}

 \theoremstyle{definition}

 \theoremstyle{remark}
 \newtheorem*{remark}{Remark}

 \numberwithin{equation}{section}

\newcommand{\CC}{{\mathbb C}}
\newcommand{\DD}{{\mathbb D}}

\newcommand{\TT}{{\mathbb T}}

\newcommand{\cD}{{\mathcal D}}
\DeclareMathOperator{\hol}{\mathrm Hol}

\begin{document}

\title[Hadamard multipliers]{Hadamard multipliers on weighted Dirichlet spaces}

\author[Mashreghi]{Javad Mashreghi}
\address{D\'epartement de math\'ematiques et de statistique, Universit\'e Laval,
Qu\'ebec City (Qu\'ebec), Canada G1V 0A6}
\email{javad.mashreghi@mat.ulaval.ca}

\author[Ransford]{Thomas Ransford}
\address{D\'epartement de math\'ematiques et de statistique, Universit\'e Laval,
Qu\'ebec City (Qu\'ebec), Canada G1V 0A6}
\email{thomas.ransford@mat.ulaval.ca}

\thanks{JM supported by an NSERC grant. TR supported by grants from NSERC and the Canada Research Chairs program.}



\date{2 July 2019}

\begin{abstract}
The Hadamard product of two power series is obtained by multiplying them coefficientwise.
In this paper we characterize those power series that act as Hadamard multipliers on
all weighted Dirichlet spaces on the disk with superharmonic weights, and we obtain sharp estimates
on the corresponding multiplier norms. Applications include an analogue of Fej\'er's theorem in these spaces, 
and a new estimate for the weighted Dirichlet integrals of dilates. 
\end{abstract}

\maketitle


\section{Introduction and statement of results}\label{S:intro}

\subsection{Weighted Dirichlet spaces}
Let  $\DD$ be the open unit disk and  $\TT$ the unit circle.
We write $\hol(\DD)$ for the set of holomorphic functions on $\DD$.
Given a positive superharmonic function $\omega$  on $\DD$ and $f\in\hol(\DD)$, we define
\begin{equation}\label{E:Domega}
\cD_\omega(f):=\int_\DD|f'(z)|^2\,\omega(z)\,dA(z),
\end{equation}
where $dA$ denotes normalized area measure on $\DD$.
The \emph{weighted Dirichlet space} $\cD_\omega$ 
is the set of $f\in\hol(\DD)$ with $\cD_\omega(f)<\infty$.
Defining
\begin{equation}\label{E:Domeganorm}
\|f\|_{\cD_\omega}^2:=|f(0)|^2+\cD_\omega(f) 
\qquad(f\in\cD_\omega),
\end{equation}
makes $\cD_\omega$ into a Hilbert space.

The  class of superharmonic weights was introduced 
and studied by Aleman \cite{Al93}. 
It includes two important subclasses:
\begin{itemize}
\item the power weights
$\omega(z):=(1-|z|^2)^\alpha ~(0\le\alpha\le 1)$,
which form a scale linking the classical Dirichlet space $\cD$
(where $\alpha=0$) to the Hardy space $H^2$
(where $\alpha=1$);
\item the harmonic weights, introduced earlier by Richter
\cite{Ri91} in connection with his analysis of shift-invariant subspaces
of the classical Dirichlet space.
\end{itemize}

\subsection{Hadamard multipliers}

Given formal power series $f(z):=\sum_{k=0}^\infty a_kz^k$ and $g(z):=\sum_{k=0}^\infty b_kz^k$,
we define their \emph{Hadamard product} to be the formal power series given by the formula
\[
(f*g)(z):=\sum_{k=0}^\infty a_kb_k z^k.
\]
Obviously, if one of $f$ or $g$ is a polynomial, then $f*g$ is a polynomial too.
Also,  if both $f,g\in\hol(\DD)$, then $f*g\in \hol(\DD)$ as well.

In this paper we study the \emph{Hadamard multipliers} of $\cD_\omega$,
namely those $h$ with the property that $h*f\in\cD_\omega$ whenever
$f\in\cD_\omega$. We also seek optimal estimates for 
$\cD_\omega(h*f)$ in terms of $\cD_\omega(f)$ and the Taylor coefficients of $h$.

\subsection{Statement of main results}

Given a sequence of complex numbers $(c_k)_{k\ge1}$, we write
$T_c$ for the infinite matrix
\begin{equation}\label{E:Tc}
T_c:=
\begin{pmatrix}
c_1 &c_2-c_1 &c_3-c_2 &c_4-c_3 &\dots\\
0 &c_2 &c_3-c_2 &c_4-c_3 &\dots\\
0 &0 &c_3 &c_4-c_3 &\dots\\
0 &0 &0 &c_4 &\dots\\
\vdots &\vdots &\vdots &\vdots &\ddots
\end{pmatrix}.
\end{equation}
This matrix may or may not act as a bounded operator on $\ell^2$.
If it does, then we write $\|T_c\|$ for its operator norm.
If not, then we write $\|T_c\|=\infty$.

If the $(c_k)$ are the  coefficients of a formal power series $h(z)=\sum_{k=0}^\infty c_kz^k$,
then we also write $T_h$ in place of $T_c$. Note that, in this situation, 
the coefficient $c_0$ plays no role.
Our main result is the following theorem.

\begin{theorem}\label{T:main}
Let $h(z)$ be a formal power series. The following statements are equivalent.
\begin{enumerate}[\normalfont (i)]
\item $h$ is a Hadamard multiplier of $\cD_\omega$ for every superharmonic weight $\omega$.
\item $T_h$ acts a bounded operator on $\ell^2$.
\end{enumerate}
In this case, for all superharmonic weights $\omega$ on $\DD$ and all $f\in\cD_\omega$, we have
\begin{equation}\label{E:main}
\cD_\omega(h*f)\le\|T_h\|^2\cD_\omega(f).
\end{equation}
The constant $\|T_h\|^2$ is best possible.
\end{theorem}

\begin{remark}
Note that, for an individual weight $\omega$, the equivalence between (i) and (ii) may well fail to hold. 
For example, the Hadamard multipliers of the classical Dirichlet space $\cD$ 
are precisely those $h(z)$ with bounded coefficients. 
In particular,  the function $h(z)=z+z^3+z^5+z^7+\dots$ is a Hadamard multiplier of $\cD$.
However, in this case, 
\[
T_h:=
\begin{pmatrix}
1 &-1 &1 &-1 &1 &-1 &\dots\\
0 &0 &1 &-1 &1 &-1 &\dots\\
0 &0 &1 &-1 &1 &-1 &\dots\\
0 &0 &0 &0 &1& -1 &\dots\\
0 &0 &0 &0 &1& -1 &\dots\\
0 &0 &0 &0 &0& 0 &\dots\\
\vdots &\vdots &\vdots &\vdots &\vdots &\vdots &\ddots
\end{pmatrix},
\]
which does not act as a bounded operator on $\ell^2$.
\end{remark}

To apply Theorem~\ref{T:main}, 
it is helpful to have at our disposal some criteria for the boundedness of the matrix $T_c$,
as well as quantitative estimates for its norm. The next theorem collects together some results of this kind.

\begin{theorem}\label{T:Tc}
Let $c:=(c_k)_{k\ge1}$ be a sequence of complex numbers, and let $T_c$ be defined by \eqref{E:Tc}.
\begin{enumerate}[{\normalfont(i)}]
\item We have the pair of estimates:
\begin{align*}
\|T_c\|&\le \sup_{k\ge1}|c_k|+2\sup_{k\ge2}k|c_k-c_{k-1}|,\\
\|T_c\|&\ge \sup_{k\ge1}\Bigl\{|c_k|^2+(k-1)|c_{k}-c_{k-1}|^2\Bigr\}^{1/2}.
\end{align*}
\item If there exists $n\ge1$ such that $c_k=0$ for all $k>n$, then
\[
\|T_c\|^2\le (n+1)\sum_{k=1}^n |c_{k+1}-c_k|^2.
\]
\item If $c_k\to0$ as $k\to\infty$, then
\[
\|T_c\|\le \sum_{k=1}^\infty \sqrt{k(k+1)}|c_{k+2}-2c_{k+1}+c_k|.
\]
\end{enumerate}
\end{theorem}

Theorems~\ref{T:main} and~\ref{T:Tc} lead  to the following corollary.

\begin{corollary}\label{C:necsuff}
For $h(z):=\sum_{k=0}^\infty c_kz^k$ to be a Hadamard multiplier of $\cD_\omega$ 
for all superharmonic weights $\omega$:
\begin{itemize}
\item a necessary condition is that $c_k=O(1)$ and $(c_k-c_{k-1})=O(1/\sqrt{k})$,
\item a sufficient condition is that $c_k=O(1)$ and $(c_k-c_{k-1})=O(1/k)$.
\end{itemize}
\end{corollary}

As a further application of Theorems~\ref{T:main} and~\ref{T:Tc},
we obtain sharp estimates for the Dirichlet, Fej\'er and de la Vall\'ee Poussin kernels
as Hadamard multipliers of $\cD_\omega$.

\begin{corollary}\label{C:kernels}
Let $n\ge0$ and let $\omega$ be a superharmonic weight.
\begin{enumerate}[\normalfont (i)]
\item If $D_n(z):=\sum_{k=0}^n z^k$, then
\[
\cD_\omega(D_n*f)\le (n+1)\cD_\omega(f) \quad(f\in\cD_\omega).
\]
\item If $K_n(z):=\sum_{k=0}^n(1-k/(n+1))z^k$, then
\[
\cD_\omega(K_n*f)\le \frac{n}{n+1}\cD_\omega(f) \quad(f\in\cD_\omega).
\]
\item If $V_n(z):=\sum_{k=0}^{n-1} z^k+\sum_{k=n}^{2n-1}(2-k/n)z^k$, then
\[
\cD_\omega(V_n*f)\le 2\cD_\omega(f) \quad(f\in\cD_\omega).
\]
\end{enumerate}
Moreover, there exists a superharmonic weight $\omega$ such that,
for each $n\ge0$, 
the constants are best possible in (i), (ii) and (iii).
\end{corollary}

Another application of Theorems~\ref{T:main} and \ref{T:Tc} is to
radial dilates. Given $f\in\hol(\DD)$ and $r\in[0,1)$, 
we define $f_r\in\hol(\DD)$ by $f_r(z):=f(rz)$. Notice that $f_r=P_r*f$, where
$P_r$ is the Poisson kernel, given by $P_r(z):=\sum_{k=0}^\infty r^kz^k$.
We obtain the following corollary.

\begin{corollary}\label{C:dilate}
Let $\omega$ be a superharmonic weight and let $f\in\cD_\omega$.
Then  $f_r\in\cD_\omega$ for all $r\in[0,1)$ and
\begin{equation}\label{E:dilate}
\cD_\omega(f_r)\le r^2(2-r)\cD_\omega(f) 
\quad(0\le r<1).
\end{equation}
\end{corollary}

Inequalities of the type $\cD_\omega(f_r)\le C\cD_\omega(f)$ have been studied by several authors,
with the value of $C$ being improved over time. Here is a brief summary of the history:
\begin{itemize}
\item $C=4$ for harmonic weights $\omega$ (Richter and Sundberg~\cite{RS91}).
\item $C=5/2$ for superharmonic weights $\omega$ (Aleman \cite{Al93}).
\item $C=2r/(1+r)\le1$ for harmonic  $\omega$  (Sarason \cite{Sa97}).
\item $C=2r/(1+r)\le1$ for superharmonic  $\omega$  (El-Fallah et al \cite{EKKMR16}).
\end{itemize}
In the last two cases, the method used was to identify certain $\cD_\omega$ 
with an appropriate de Branges--Rovnyak space, and
prove the desired inequality in that space. Our proof of \eqref{E:dilate} is direct,
and our constant $C$ is better.

One reason for studying inequalities of the type \eqref{E:dilate}
is that they can be used to prove that polynomials are dense in $\cD_\omega$.
This fact was originally established by Richter \cite{Ri91} (for harmonic $\omega$)
and  Aleman \cite{Al93} (for general superharmonic $\omega$) using a different technique,
based on a certain type of wandering-subspace theorem. 

We end this section by stating an analogue of Fej\'er's theorem for $\cD_\omega$,
which yields another, direct proof of the density of polynomials in $\cD_\omega$.
Given a power series $f(z):=\sum_{k=0}^\infty a_kz^k$, we write
\[
s_n(f)(z):=\sum_{k=0}^n a_kz^k
\quad\text{and}\quad
\sigma_n(f)(z):=\frac{1}{n+1}\sum_{k=0}^n s_k(f)(z).
\]

\begin{theorem}\label{T:Fejer}
\begin{enumerate}[\normalfont(i)]
\item If $\omega$ is a superharmonic weight and if $f\in\cD_\omega$, 
then $\|\sigma_n(f)-f\|_{\cD_\omega}\to0$
as $n\to\infty$.
\item There exist a superharmonic weight and a function $f\in\cD_\omega$ 
such that $\|s_n(f)-f\|_{\cD_\omega}\not\to0$ as $n\to\infty$.
\end{enumerate}
\end{theorem}

Part~(ii) was already known (see e.g.\ \cite[Exercise~7.3.2]{EKMR14}).
We include it to complement part~(i), 
which we believe to be new. 
 
The rest of the paper is structured as follows. 
In \S\ref{S:superharmonic} we review some basic properties of Dirichlet spaces 
with superharmonic weights. The theorems and corollaries stated above are proved
in  \S\S\ref{S:main}--\ref{S:Fejer}.
We conclude in \S\ref{S:conclusion} with some remarks and questions.


\section{Background on superharmonic weights}\label{S:superharmonic}

Let $\omega$ be a positive superharmonic function on $\DD$.
In this section we summarize some basic properties
of the weighted Dirichlet space $\cD_\omega$.
For detailed proofs and further information, we refer to \cite{Al93}.

By standard results from potential theory,
 $\omega$ is locally integrable on~$\DD$,
and $(1/r^2)\int_{|z|\le r}\omega\,dA$ is a decreasing 
function of $r$ for $0<r<1$
(see e.g.\ \cite[Theorems 2.5.1 and 2.6.8]{Ra95}. 
It follows that $\omega\in L^1(\DD)$,
and thus $\cD_\omega$ contains the polynomials.

As $\omega$ is a positive superharmonic function, 
there exists a unique positive finite Borel measure $\mu$ 
on $\overline{\DD}$ such that, for all $z\in\DD$,
\begin{equation}\label{E:murep}
\omega(z)=\int_\DD \log\Bigl|\frac{1-\overline{\zeta}z}{\zeta-z}\Bigr|\frac{2}{1-|\zeta|^2}\,d\mu(\zeta)
+\int_\TT\frac{1-|z|^2}{|\zeta-z|^2}\,d\mu(\zeta)
\end{equation}
(see e.g.\ \cite[Theorem~4.5.1]{Ra95}).
When $\mu=\delta_\zeta$, we write $\cD_\zeta$ for $\cD_\omega$.
Thus, for $f\in\hol(\DD)$, 
\[
\cD_\zeta(f)=
\begin{cases}
\displaystyle\int_\DD \log\Bigl|\frac{1-\overline{\zeta}z}{\zeta-z}\Bigr|\frac{2}{1-|\zeta|^2}|f'(z)|^2\,dA(z), &\zeta\in\DD,\\
\displaystyle\int_\DD\frac{1-|z|^2}{|\zeta-z|^2}|f'(z)|^2\,dA(z), &\zeta\in\TT.
\end{cases}
\]
For general $\omega$, corresponding to a measure $\mu$ on $\overline{\DD}$,
we can recover $\cD_\omega(f)$  from $\cD_\zeta(f)$
via Fubini's theorem: 
\begin{equation}\label{E:average}
\cD_\omega(f)=\int_{\overline{\DD}}\cD_\zeta(f)\,d\mu(\zeta).
\end{equation}

The spaces $\cD_\zeta$ are sometimes called \emph{local Dirichlet spaces}.
The following result gives a very simple alternative description of them.
Recall that $H^2$ denotes the Hardy space.

\begin{theorem}\label{T:localD}
Let $\zeta\in\overline{\DD}$. Then $f\in\cD_\zeta$ if and only if $f(z)=a+(z-\zeta)g(z)$, 
where $g\in H^2$ and $a\in\CC$. In this case $\cD_\zeta(f)=\|g\|_{H^2}^2$.
\end{theorem}

Thus, if $\zeta\in\DD$, then $\cD_\zeta$ is just $H^2$ with a different (but equivalent) norm.
But if $\zeta\in\TT$, then $\cD_\zeta$ is a proper subspace of $H^2$.
For a proof of this result, see \cite[Chapter~IV, \S1]{Al93}.


\section{Proof of Theorem~\ref{T:main}}\label{S:main}

\begin{proof}[Proof that (ii)$\Rightarrow$(i)]
Let $h(z):=\sum_{k=0}^\infty c_kz^k$ be a power series such that $\|T_h\|<\infty$. 
It  follows that the sequence $(c_{k+1}-c_{k})_{k\ge1}$ belongs to $\ell^2$.

Let $\zeta\in\overline{\DD}$, and let $f\in\cD_\zeta$.
By Theorem~\ref{T:localD}, there exist $g\in H^2$ and $a\in\CC$
such that $f(z)=a+(z-\zeta)g(z)$, and $\cD_\zeta(f)=\|g\|_{H^2}^2$.
Writing $f(z):=\sum_{k=0}^\infty a_kz^k$ and $g(z):=\sum_{k=0}^\infty b_kz^k$,  
and equating coefficients of $z^k$ in the resulting formula, we obtain the relations
\begin{equation}\label{E:coeffs}
\begin{cases}
a_0&=a- b_0\zeta,\\
a_k&=b_{k-1}- b_k\zeta \quad(k\ge1).
\end{cases}
\end{equation}

As both the sequences $(b_k)_{k\ge1}$ and  $(c_{k+1}-c_k)_{k\ge1}$ belong to $\ell^2$,
the series $\sum_{k=1}^\infty (c_{k+1}-c_k)b_k$ converges absolutely,
and it thus makes sense to define a formal power series $F$ by
\[
F(z):=\sum_{j=0}^\infty \Bigl(c_{j+1}b_j+\sum_{k=j+1}^\infty (c_{k+1}-c_k)b_k\zeta^{k-j}\Bigr)z^j.
\]
If we multiply this power series by $(z-\zeta)$, then we obtain a new power series whose $j$-th coefficient
(for $j\ge1$) is given by
\begin{align*}
c_jb_{j-1}&+\sum_{k=j}^\infty (c_{k+1}-c_k)b_k\zeta^{k-j+1}
-c_{j+1}b_j\zeta-\sum_{k=j+1}^\infty(c_{k+1}-c_k)b_k\zeta^{k-j+1}\\
&=c_j(b_{j-1}-b_j\zeta).
\end{align*}
In view of the relations \eqref{E:coeffs}, this is exactly equal to $c_ja_j$. In
other words, we have $(h*f)(z)=(z-\zeta)F(z)+A$, where $A$ is a constant.

We claim that $F\in H^2$. Indeed,
\begin{align*}
\|F\|_{H^2}^2
&=\sum_{j=0}^\infty \Bigl|c_{j+1}b_j+\sum_{k=j+1}^\infty (c_{k+1}-c_k)b_k\zeta^{k-j}\Bigr|^2\\
&\le \sup_{|\eta|=1}\sum_{j=0}^{\infty}\Bigl|c_{j+1}b_j+\sum_{k=j+1}^\infty (c_{k+1}-c_k)b_k\eta^{k-j}\Bigr|^2
&&\text{(max.\ principle)}\\
&=\sup_{|\eta|=1}\sum_{j=0}^{\infty}\Bigl|c_{j+1}b_j\eta^j+\sum_{k=j+1}^\infty (c_{k+1}-c_k)b_k\eta^k\Bigr|^2,
\end{align*}
and so, writing $v_\eta:=(b_j\eta^j)_{j\ge0}$, we have
\begin{align*}
\|F\|_{H^2}^2&\le \sup_{|\eta|=1}\|T_h(v_\eta)\|_{\ell^2}^2=\sup_{|\eta|=1}\|T_h\|^2\|v_\eta\|_{\ell^2}^2\\
&=\|T_h\|^2\sum_{k=0}^\infty|b_k|^2=\|T_h\|^2\|g\|_{H^2}^2=\|T_h\|^2\cD_\zeta(f).
\end{align*}

In combination with Theorem~\ref{T:localD}, these observations imply that $h*f\in\cD_\zeta$ and that
\[
\cD_\zeta(h*f)=\|F\|_{H^2}^2\le \|T_h\|^2\cD_\zeta(f).
\]

Now let $\omega$ be a superharmonic weight on $\DD$, 
and let $f\in\cD_\omega$.
Let $\mu$ be the associated measure on $\overline{\DD}$
so that \eqref{E:murep} holds. Formula \eqref{E:average} then gives
\[
\cD_\omega(f)=\int_{\overline{\DD}}\cD_\zeta(f)\,d\mu(\zeta).
\]
Thus $f\in\cD_\zeta$ for $\mu$-almost every $\zeta\in\overline{\DD}$,
and, by what we have proved above, $\cD_\zeta(h*f)\le \|T_h\|^2\cD_\zeta(f)$ for all such $\zeta$. 
Consequently,
\[
\cD_\omega(h*f)
=\int_{\overline{\DD}}\cD_\zeta(h*f)\,d\mu(\zeta)
\le \|T_h\|^2\int_{\overline{\DD}}\cD_\zeta(f)\,d\mu(\zeta)
\le \|T_h\|^2\cD_\omega(f).
\]
Thus $h$ is a Hadamard multiplier of  $\cD_\omega$,
and \eqref{E:main} holds.
\end{proof}

\begin{proof}[Proof that  (i)$\Rightarrow$(ii)]
Consider the closed subspace of $\cD_1$ defined by
\[
\cD_1^0:=\{f\in\cD_1: f(0)=0\}.
\]
Note that $\|f\|_{\cD_1}^2=\cD_1(f)$ for all $f\in\cD_1^0$.
Also, by Theorem~\ref{T:localD}, we have $f\in\cD_1^0$ if and only if
there exists $g\in H^2$ such that $f(z)=g(0)+(z-1)g(z)$, 
and in this case $\cD_1(f)=\|g\|_{H^2}^2$. Thus, if we define $U:H^2\to\cD_1^0$
by 
\[
(Ug)(z):=g(0)+(z-1)g(z) \quad(g\in H^2),
\]
then $U$ is an isometry of $H^2$ onto $\cD_1^0$,
i.e., $U$ is a unitary operator.

Now let $h(z):=\sum_{k=0}^\infty c_kz^k$ be a Hadamard multiplier of $\cD_1$.
Clearly it is also a Hadamard multiplier of $\cD_1^0$.
Define $M_h:\cD_1^0\to\cD_1^0$ by 
\[
M_h(f):=h*f \quad(f\in\cD_1^0).
\]
An application of the closed graph theorem shows that $M_h$
is a bounded linear map of $\cD_1^0$ into itself. Hence
$U^*M_hU$ is a bounded linear map of $H^2$ into itself.
We shall show that the matrix of this map with respect to the standard 
orthonormal basis $\{1,z,z^2,\dots\}$ of $H^2$ is exactly $T_h$.

Fix $k\ge0$.  Then, identifying $H^2$ with $\ell^2$ in the standard way, we have
\begin{align*}
UT_h(z^k)
&=U\Bigl((c_{k+1}-c_k)(1+z+\dots+ z^{k-1})+c_{k+1}z^k\Bigr)\\
&=(c_{k+1}-c_k)+(c_{k+1}-c_k)(z^k-1)+c_{k+1}(z^{k+1}-z^k)\\
&=c_{k+1}z^{k+1}-c_kz^k=M_h(z^{k+1}-z^k)=M_hU(z^k).
\end{align*}
Thus $U^*M_hU(z^k)=T_h(z^k)$ for all $k\ge0$,
as claimed.
 
It follows that $T_h$ acts as a bounded linear operator on $\ell^2$. 
Also, we have
\[
\|T_h\|^2=\|U^*M_hU\|^2=\|M_h\|^2=\sup\{\cD_1(f): f\in\cD_1^0,~\cD_1(f)=1\},
\]
which shows that the constant $\|T_h\|^2$ in \eqref{E:main} is best possible.
\end{proof}

We note in passing the following consequence
of the proof above.

\begin{corollary}
A power series $h$ is a Hadamard multiplier of $\cD_1$
if and only if it is a Hadamard multiplier of $\cD_\omega$ for all superharmonic weights $\omega$.
\end{corollary}


\section{Proofs of Theorem~\ref{T:Tc}\,(i) and Corollary~\ref{C:necsuff}}\label{S:Tc(i)}

\begin{proof}[Proof of Theorem~\ref{T:Tc}\,(i)]
We can decompose $T_c$ as
\[
T_c=
\begin{pmatrix}
c_1 &0 &0 &0 &\dots\\
0 &c_1 &0 &0 &\dots\\
0 &0 &c_2 &0 &\dots\\
0 &0 &0 &c_3 &\dots\\
\vdots &\vdots &\vdots &\vdots &\ddots
\end{pmatrix}
+
\begin{pmatrix}
0            &c_2-c_1 &c_3-c_2 &c_4-c_3 &\dots\\
0            &c_2-c_1 &c_3-c_2 &c_4-c_3 &\dots\\
0            &0            &c_3-c_2 &c_4-c_3 &\dots\\
0            &0            &0            &c_4-c_3 &\dots\\
\vdots &\vdots &\vdots &\vdots &\ddots
\end{pmatrix}.
\]
The first matrix has norm at most $\sup_{k\ge1}|c_k|$.
As for the second matrix, the absolute value of each entry is at most
$\sup_{k\ge2}k|c_k-c_{k-1}|$ times the corresponding entry in the Ces\`aro matrix,
\[
C:=
\begin{pmatrix}
1 &1/2 &1/3 &1/4 &\dots\\
0 &1/2 &1/3 &1/4 &\dots\\
0 &0    &1/3 &1/4 &\dots\\
0 &0    &0    &1/4 &\dots\\
\vdots &\vdots &\vdots &\vdots &\ddots
\end{pmatrix}.
\]
It is well known that $\|C\|=2$ (see e.g.\ \cite[pp.~128--129]{BHS65}). It follows that
\[
\|T_c\|\le \sup_{k\ge1}|c_k|+2\sup_{k\ge2}k|c_k-c_{k-1}|.
\]
This is the desired upper bound for $\|T_c\|$. The lower bound is obtained 
by noting that $\|T_c\|\ge \sup_{k\ge1}\|T_c(e_k)\|_{\ell^2}$, 
where $(e_k)$ is the standard unit
vector basis of $\ell^2$.
\end{proof}

\begin{proof}[Proof of Corollary~\ref{C:necsuff}]
This corollary follows directly from Theorems~\ref{T:main} and~\ref{T:Tc}\,(i).
\end{proof}


\section{Proofs of Theorem~\ref{T:Tc}\,(ii) and Corollary~\ref{C:kernels}}\label{S:Tc(ii)}

\begin{proof}[Proof of Theorem~\ref{T:Tc}\,(ii)]
Let $n\ge1$, and suppose that $c_k=0$ for all $k>n$.
Let $\xi:=(\xi_1,\xi_2,\xi_3,\dots)\in\ell^2$.
Then we have
\begin{align*}
\|T_c(\xi)\|_{\ell^2}^2
&=\sum_{j=1}^{n}\Bigl|c_{j}\xi_j+\sum_{k=j}^n (c_{k+1}-c_{k})\xi_{k+1}\Bigr|^2\\
&=\sum_{j=1}^{n}\Bigl|\sum_{k=j}^{n} (c_{k+1}-c_{k})(\xi_{k+1}-\xi_j)\Bigr|^2\\
&\le\sum_{j=1}^{n}\Bigl(\sum_{k=j}^n |c_{k+1}-c_k|^2\Bigr)\Bigl(\sum_{k=j}^n |\xi_{k+1}-\xi_j|^2\Bigr)\\
&\le\Bigl(\sum_{k=1}^n |c_{k+1}-c_k|^2\Bigr)\sum_{j=1}^{n}\sum_{k=j}^n |\xi_{k+1}-\xi_j|^2.
\end{align*}
Now, 
\[
\sum_{j=1}^{n}\sum_{k=j}^n |\xi_{k+1}-\xi_j|^2
=\frac{1}{2}\sum_{p=1}^{n+1}\sum_{q=1}^{n+1}|\xi_p-\xi_q|^2
=(n+1)\sum_{k=1}^{n+1} |\xi_p|^2-\Bigl|\sum_{k=1}^{n+1} \xi_p\Bigr|^2.
\]
The right-hand side is at most $(n+1)\sum_{k=1}^{n+1}|\xi_k|^2\le (n+1)\|\xi\|_{\ell^2}^2$.
Hence
\[
\|T_c(\xi)\|_{\ell^2}^2\le (n+1)\sum_{k=1}^n|c_{k+1}-c_k|^2\|\xi\|_{\ell^2}^2.
\]
This shows that $\|T_c\|^2\le (n+1)\sum_{k=1}^n|c_{k+1}-c_k|^2$, completing the proof.
\end{proof}

\begin{proof}[Proof of Corollary~\ref{C:kernels}]
The inequalities in (i), (ii) and (iii) all follow directly from
Theorems~\ref{T:main} and~\ref{T:Tc}\,(ii).
It remains to justify the sharpness of the constants.

(i) Let $f(z):=nz^{n+1}-(n+1)z^n+1=(z-1)(nz^n-z^{n-1}-z^{n-2}-\dots-1)$.
Then $(D_n*f)(z)=-(n+1)z^n+1=-n-(n+1)(z-1)(z^{n-1}+\dots+z+1)$.
By Theorem~\ref{T:localD}, we have 
\[
\cD_1(f)=n(n+1)
\quad\text{and}\quad
\cD_1(D_n*f)=(n+1)^2n=(n+1)\cD_1(f).
\]

(ii) Let $f(z):=z^{n+1}-(n+1)z+n=(z-1)(z^n+z^{n-1}+\cdots+z-n)$.
Then $(K_n*f)(z)=n-nz=-n(z-1)$. By Theorem~\ref{T:localD} again,  
\[
\cD_1(f)=n(n+1)
\quad\text{and}\quad
\cD_1(K_n*f)=n^2=\frac{n}{n+1}\cD_1(f).
\]

(iii) Let $f(z):=z^{2n}-2z^n+1=(z-1)(z^{2n-1}+\dots+z^n-z^{n-1}-\dots-1)$.
Then $(V_n*f)(z)=-2z^n+1=-1-2(z-1)(z^{n-1}+\dots+z+1)$. 
Once more, using Theorem~\ref{T:localD},
we have
\[
\cD_1(f)=2n
\quad\text{and}\quad
\cD_1(V_n*f)=4n=2\cD_1(f).
\]

Thus  the constants in (i), (ii) and (iii) are sharp for all $n\ge0$.
\end{proof}


\section{Proofs of Theorem~\ref{T:Tc}\,(iii) and Corollary~\ref{C:dilate}}\label{S:Tc(iii)}
We shall prove Theorem~\ref{T:Tc}\,(iii) using a superposition technique
enshrined in the following lemma.

\begin{lemma}\label{L:superposition}
Let $(h_n)_{n\ge0}$ be a sequence of formal power series such that
$\sum_{n=0}^\infty|h_n(0)|<\infty$ and
$\sum_{n=0}^\infty\|T_{h_n}\|<\infty$.
Then the formal power series $h:=\sum_{n=0}^\infty h_n$ converges coefficientwise,
and $T_h$ acts as a bounded operator on $\ell^2$, with
\begin{equation}\label{E:superposition}
\|T_h\|\le \sum_{n=0}^\infty \|T_{h_n}\|.
\end{equation}
\end{lemma}

\begin{proof}
For each pair of integers $r \le s$, we have
\begin{equation}\label{E:rs}
\Bigl\|\sum_{n=r}^sT_{h_n}\Bigr\|
\le \sum_{n=r}^{s}\|T_{h_n}\|.
\end{equation}
It follows that the series $\sum_{n=0}^\infty T_{h_n}$ is Cauchy in  
the Banach space of bounded  operators on $\ell^2$,
hence convergent in that space, to $T$ say.
As convergence in the space of bounded operators on $\ell^2$
implies the entrywise convergence of the corresponding matrices,
it follows that $\sum_{n=0}^\infty h_n$ converges coefficientwise to $h$ and 
that $T=T_h$. (Note that, since the constant coefficients of $h_n$ do not appear in $T_{h_n}$,
we need to assume $\sum_n|h_n(0)|<\infty$ separately.) Thus $T_h$ acts as a bounded operator on $\ell^2$.
Finally, letting $r=0$ and $s\to\infty$ in \eqref{E:rs}, we obtain \eqref{E:superposition}.
\end{proof}

\begin{proof}[Proof of Theorem~\ref{T:Tc}\,(iii)]
Let $(c_n)_{n\ge1}$ be a sequence such that
\[
\lim_{n\to\infty}c_n=0 
\quad\text{and}\quad
\sum_{n=1}^\infty\sqrt{n(n+1)}|c_{n+2}-2c_{n+1}+c_n|<\infty.
\]
Set $c_0:=0$. For each $n\ge0$,
let $h_n:=\lambda_nK_{n}$,
where
\[
\lambda_n:=(n+1)(c_{n+2}-2c_{n+1}+c_{n})
\quad\text{and}\quad
K_n(z):=\sum_{k=0}^{n}(1-k/(n+1))z^k.
\]
Then $K_{n}(0)=1$ for each $n$, 
and   $\|T_{K_{n}}\|=\sqrt{n/(n+1)}$
by Corollary~\ref{C:kernels}\,(ii).
Since $\sum_{n=0}^\infty|\lambda_n|<\infty$, 
we may apply Lemma~\ref{L:superposition} to
deduce that $\sum_{n=0}^\infty \lambda_n K_{n}$
converges coefficientwise to a power series $h$ such that
\[
\|T_h\|\le \sum_{n=0}^\infty |\lambda_n|\|T_{K_{n}}\|=\sum_{n=1}^\infty\sqrt{n(n+1)}|c_{n+2}-2c_{n+1}+c_n|.
\]
We claim  that $h(z)=\sum_{k=0}^\infty c_kz^k$.
If so, then 
\[
\|T_c\|=\|T_h\|\le \sum_{n=1}^\infty\sqrt{n(n+1)}|c_{n+2}-2c_{n+1}+c_n|,
\]
and the theorem is proved.

It remains to justify the claim. We have
\[
\sum_{n=0}^\infty\lambda_nK_n(z)
=\sum_{n=0}^\infty\lambda_n\sum_{k=0}^{n}\Bigl(1-\frac{k}{n+1}\Bigr)z^k
=\sum_{k=0}^\infty\Bigl(\sum_{n=k}^\infty\lambda_n\Bigl(1-\frac{k}{n+1}\Bigr)\Bigr)z^k.
\]
Thus $h(z)=\sum_{k=0}^\infty d_kz^k$, where
\[
d_k:=\sum_{n=k}^\infty\lambda_n\Bigl(1-\frac{k}{n+1}\Bigr)
\quad(k\ge0).
\]
Now, for each $k\ge0$, we have
\[
d_{k}-d_{k+1}=\sum_{n=k+1}^\infty\frac{\lambda_n}{n+1}, 
\]
and so
\[
(d_{k}-d_{k+1})-(d_{k+1}-d_{k+2})=\frac{\lambda_k}{k+1}=c_{k+2}-2c_{k+1}+c_{k}.
\]
It follows that $u_k:=c_k-d_k$ satisfies the relation $u_{k+2}-2u_{k+1}+u_{k}=0$ for all $k\ge0$.
Therefore $u_k=\alpha k+\beta$ for some constants $\alpha,\beta$. 
Now $c_k\to0$ by hypothesis, and $d_k\to0$ by its very definition, 
so $u_k\to0$, and thus necessarily $\alpha=\beta=0$. 
Hence $d_k=c_k$ for all $k$, and so finally $h(z)=\sum_{k=0}^\infty c_kz^k$,
as claimed.
\end{proof}

\begin{proof}[Proof of Corollary~\ref{C:dilate}]
As remarked in the introduction, 
we have $f_r=P_r*f$, where $P_r(z):=\sum_{k=0}^\infty r^kz^k$,
the Poisson kernel. By Theorems~\ref{T:main} and ~\ref{T:Tc}\,(iii), 
we have $\cD_\omega(P_r*f)\le \|T_{P_r}\|^2\cD_\omega(f)$,
where
\begin{align*}
\|T_{P_r}\|&\le\sum_{k=1}^\infty \sqrt{k(k+1)}|r^{k+2}-2r^{k+1}+r^{k}|\\
&=(1-r)^2\sum_{k=1}^\infty \sqrt{k(k+1)}r^{k}\\
&\le (1-r)^2\Bigl(\sum_{k=1}^\infty k r^k\Bigr)^{1/2}\Bigl(\sum_{k=1}^\infty (k+1)r^k\Bigr)^{1/2}\\
&=r(2-r)^{1/2},
\end{align*}
the last inequality arising from the Cauchy--Schwarz inequality.
\end{proof}


\section{Proof of Theorem~\ref{T:Fejer} }\label{S:Fejer}

Note that $s_n(f)$ and $\sigma_n(f)$ can be written as Hadamard products,
namely
\[
s_n(f)=D_n*f 
\quad\text{and}\quad
\sigma_n(f)=K_n*f,
\]
where $D_n$ and $K_n$ are the Dirichlet and Fej\'er kernels already studied
in Corollary~\ref{C:kernels}. We make key use of this remark in what follows.

\begin{proof}[Proof of Theorem~\ref{T:Fejer}]
(i) By the parallelogram identity,  we have
\[
\cD_\omega(\sigma_n(f)-f)+\cD_\omega(\sigma_n(f)+f)=2\cD_\omega(\sigma_n(f))+2\cD_\omega(f).
\]
Corollary~\ref{C:kernels}\,(ii) gives
\[
\cD_\omega(\sigma_n(f))=\cD_\omega(K_n*f)\le\cD_\omega(f) \quad(n\ge0).
\]
Also,  $\sigma_n(f)\to f$ locally uniformly on $\DD$,
so, by \eqref{E:Domega} and Fatou's lemma,
\[
\liminf_{n\to\infty}\cD_\omega(\sigma_n(f)+f)\ge 4\cD_\omega(f).
\]
It follows that 
\[
\limsup_{n\to\infty}\cD_\omega(\sigma_n(f)-f)\le 2\cD_\omega(f)+2\cD_\omega(f)-4\cD_\omega(f)=0.
\]
Finally, since  $\sigma_n(f)(0)\to f(0)$, we conclude that
\[
\|\sigma_n(f)-f\|_{\cD_\omega}^2=|\sigma_n(f)(0)-f(0)|^2+\cD_\omega(\sigma_n(f)-f)\to0.
\]

(ii) We consider $f\mapsto s_n(f)$ as a linear map  $:\cD_1\to\cD_1$.
The calculations in the proof of Corollary~\ref{C:kernels} show that, 
if  $f_n(z):=nz^{n+1}-(n+1)z^n$, then 
\begin{align*}
\|f_n\|_{\cD_1}^2&=\cD_1(f_n)=n(n+1),\\
\|s_n(f_n)\|_{\cD_1}^2&=\|D_n*f_n\|_{\cD_1}^2=(n+1)^2n.
\end{align*}
Thus the norm of $s_n$ as a linear operator satisfies $\|s_n\|\ge\sqrt{n+1}$.
Hence $\sup_n\|s_n\|=\infty$. By the Banach--Steinhaus theorem,
there exists $f\in\cD_1$ such that $\sup_n\|s_n(f)\|_{\cD_1}=\infty$.
In particular, $\|s_n(f)-f\|_{\cD_1}\not\to0$.
\end{proof}


\section{Concluding remarks and questions}\label{S:conclusion}

(1) Once it is known that polynomials are dense in $\cD_\omega$,  it is a routine matter to prove that
the following statements are equivalent:
\begin{itemize}
\item
$\lim_{n\to\infty}\|h_n*f-h*f\|_{\cD_\omega}=0$ for all superharmonic  $\omega$ and all $f\in\cD_\omega$.
\item
$h_n\to h$ coefficientwise and $\sup_n\|T_{h_n}\|<\infty$.
\end{itemize}
Evidently, this generalizes Theorem~\ref{T:Fejer}. However, the interest of Theorem~\ref{T:Fejer}
lies in the fact that, since its  proof makes no appeal to the density of polynomials, it can be used to 
give an independent proof of this density.

\medskip
(2) Although \eqref{E:dilate} provides an estimate of the constant $C$ such that
$\cD_\omega(f_r)\le C\cD_\omega(f)$ that is better than those known to date, it is clearly still not optimal.
Indeed, we know that best constant is exactly $\|T_{P_r}\|^2$. Is there a `nice' algebraic expression for this norm?

\medskip
(3) Which sequences $(c_k)$ have the property that $\|T_c\|<\infty$? There are various equivalent ways to reformulate this question. For example, which real sequences $(a_k)$ have the property that the self-adjoint, L-shaped matrix
\[
\begin{pmatrix}
a_1 &a_2 &a_3 &a_4 &\dots\\
a_2 &a_2 &a_3 &a_4 &\dots\\
a_3 &a_3 &a_3 &a_4 &\dots\\
a_4 &a_4 &a_4 &a_4 &\dots\\
\vdots &\vdots &\vdots &\vdots &\ddots
\end{pmatrix}
\]
acts as a bounded operator on $\ell^2$?

\medskip
(4) Can we characterize the Hadamard multipliers $h(z):=\sum_{k=0}^\infty c_kz^k$ of $\cD_\omega$ for a fixed superharmonic weight $\omega$ on $\DD$? We know that:
\begin{itemize}
\item if $\zeta\in\DD$, then $h$ is a Hadamard multiplier of $\cD_\zeta \iff \sup|c_k|<\infty$,
\item if $\zeta\in\TT$, then $h$ is a Hadamard multiplier of $\cD_\zeta \iff \|T_c\|<\infty$.
\end{itemize}
Are there examples of superharmonic weights $\omega$ for which the characterization lies somewhere strictly between these two extremes?


\bibliographystyle{plain}
\bibliography{manuscript}


\end{document}